\documentclass[final]{l4dc2020} 
\usepackage{times}

\usepackage{amsfonts, bbm, bm}
\usepackage{mathtools}
\usepackage{enumitem}

\makeatletter
\newcommand*{\rom}[1]{\expandafter\@slowromancap\romannumeral #1@}
\makeatother
\newcommand\EnumPrefix{}
\newlist{senenum}{enumerate}{10}
\setlist[senenum]{label=\arabic*.,ref=\EnumPrefix,leftmargin=*}

\newcommand{\remove}[1]{}

\newcommand{\bmat}[1]{\begin{bmatrix}#1\end{bmatrix}}

\newcommand{\Prob}[1]{\mathbb{P} \left( #1 \right)}
\newcommand{\Exp}[1]{\mathbb{E} \left[ #1 \right]}
\newcommand{\Expx}[1]{\mathbb{E}_{x_0 \sim \mathcal{D}} \left[ #1 \right]}
\newcommand{\Expxt}[1]{\mathbb{E}_{x_0 \sim \mathcal{D}, \omega_0\sim \rho} \left[ #1 \right]}

\newcommand{\1}[1]{\mathbf{1}_{#1}}
\newcommand{\eps}{\varepsilon}
\newcommand{\Nor}{\mathcal{N}}
\newcommand{\N}{\mathcal{N}}

\newcommand{\R}{\mathbb{R}}

\newcommand{\tr}[1]{\textup{tr}\! \left( #1 \right)}
\newcommand{\diag}[1]{\textup{diag}\! \left( #1 \right)}
\newcommand{\vect}[1]{\textup{vec}\! \left( #1 \right)}

\newcommand{\sumtinf}{\sum_{t=0}^\infty}

\newcommand{\EPnext}{\mathcal{E}_i(P)}
\newcommand{\EPKnext}{\mathcal{E}_i(P^{\hat{K}})}

\newcommand{\Enext}[1]{\mathcal{E}_i(#1)}

\newcommand{\EPhatnext}{\hat{\mathcal{E}}(P^{\hat{K}})}

\newcommand{\deltaK}{\Delta \hat{K}}
\newcommand{\deltaKi}{\Delta K_i}

\title[Policy Learning of MDPs with Mixed Continuous/Discrete Variables]{Policy Learning of MDPs with Mixed Continuous/Discrete Variables: \\A Case Study on  Model-Free Control of Markovian Jump Systems
}

\author{%
 \Name{Joao P. {Jansch-Porto}} \Email{janschp2@illinois.edu}\\
 \Name{Bin Hu} \Email{binhu7@illinois.edu}\\
 \Name{Geir E. Dullerud} \Email{dullerud@illinois.edu}\\
 \addr Coordinated Science Laboratory, University of Illinois at Urbana-Champaign, Urbana, Illinois 61801 USA%
}

\begin{document}

\maketitle

\begin{abstract}

Markovian jump linear systems (MJLS) are an important class of dynamical systems that arise in many control applications. 
In this paper, we introduce the problem of controlling unknown (discrete-time) MJLS as a new benchmark for policy-based reinforcement learning of Markov decision processes (MDPs) with mixed continuous/discrete state variables. 
Compared with the traditional linear quadratic regulator (LQR), our proposed problem leads to a special hybrid MDP (with mixed continuous and discrete variables) and poses significant new challenges due to the appearance of an underlying Markov jump parameter governing the mode of the system dynamics. 
Specifically, the state of a MJLS does not form a Markov chain and hence one cannot study the MJLS control problem as a MDP with solely continuous state variable. 
However, one can augment the state and the jump parameter to obtain a MDP with a mixed continuous/discrete state space.
We discuss how control theory sheds light on the policy parameterization of such hybrid MDPs. 
Then we modify the widely used natural policy gradient method to directly learn the optimal state feedback control policy for MJLS without identifying either the system dynamics or the transition probability of the switching parameter. 
We implement the (data-driven) natural policy gradient method on different MJLS examples. 
Our simulation results suggest that the natural gradient method can efficiently learn the optimal controller for MJLS with unknown dynamics. 

\end{abstract}

\section{Introduction}
\label{sec:intro}

Reinforcement learning (RL)~\citep{sutton2018reinforcement} provides a powerful framework for solving Markov decision process (MDP) problems.
Although deep RL has achieved promising empirical successes in a variety of applications~\citep{schulman2015high, levine2016end}, how to choose RL algorithms \citep{duan2016benchmarking, kakade2002natural, schulman2015trust, peters2008natural, schulman2017proximal} for a specific task is still not fully understood~\citep{henderson2018deep,rajeswaran2017towards}.
This motivates many recent research efforts on understanding the performances of RL algorithms on simplified benchmarks. 
For example, many control applications lead to MDP formulations with continuous state/action spaces, and hence
the classic Linear Quadratic Regulator (LQR) problem has been revisited as a benchmark for understanding the performance of various model-free or model-based RL algorithms on such MDPs~\citep{pmlr-v80-fazel18a,dean2017sample,malik2018derivative,tu2018gap,dean2018regret,abbasi2011regret,abbasi2018regret,yang2019global}.

Another important class of MDPs involve mixed continuous/discrete state variables \citep{boyan2001exact,toussaint2006probabilistic,guestrin2004solving}.
In this paper, our key point is that the problem of controlling unknown (discrete-time) Markov Jump Linear Systems (MJLS) \citep{costa2006discrete} provides  a simple meaningful benchmark for understanding the performances of policy-based RL algorithms
on MDPs with mixed continuous/discrete variables.
MJLS is an important class of dynamical systems that find many applications in control~\citep{bar1993estimation,fox2011tsp,hamsa2016cdc,Pavlovic2000LearningSL,sworder1999estimation,varga2013ijrnc}, and machine learning \citep{hu2017unified,hu2019characterizing}.
Notice that the state/input matrices of a MJLS are functions of a jump parameter that is typically sampled from a Markov chain with discrete state variables. 
Understanding the performance of various RL methods on the MJLS control problem  can bring many useful insights.
An important fact is that the state of a MJLS does not form a Markov chain and hence one cannot study the MJLS control problem as a MDP with solely continuous state variable. 
However, if one augments the state and the jump parameter together, a MDP with a mixed continuous/discrete state space is naturally obtained.
Therefore, we believe the optimal control of unknown MJLS is a meaningful benchmark for further understanding of mixed MDPs.

Although the MDP formulation for MJLSs involves mixed continuous/discrete state variables, our study demonstrates that we can still modify policy-based RL algorithms, such as REINFORCE \citep{williams1992simple,sutton2000policy}, to efficiently solve this problem.
The key here is to parameterize the control policy based on MJLS control theory.
Recently, model-based policy optimization methods have been shown to provably converge to the global optimal policy for MJLS control~\citep{joaoACC}. 
In this paper, we discuss how to efficiently implement these methods in a data-driven manner.
We implement the model-free natural policy gradient (NPG) method on various MJLS examples. 
Our simulation results suggest that the NPG method with the REINFORCE policy gradient estimator and a simple average baseline can learn the optimal control for an MJLS without identifying either the system dynamics or the transition probability of the jump parameter. 
This confirms that policy optimization may provide a promising solution for controlling unknown jump systems. 
It is our hope that our study serves a first step towards more understanding of RL algorithms on MDPs with mixed continuous/discrete variables.

\paragraph{Related work.}
Some previous work has studied how to apply RL methods to MJLSs with known state/input matrices and unknown jump parameter distribution~\citep{costa2002monte, beirigo2018online}. 
In our paper, both state/input matrices and the transition probability of the jump parameter are assumed to be unknown.
In addition, the continuous-time setup has also been investigated recently~\citep{he2019reinforcement}.
Our paper considers the standard discrete-time formulation of MJLSs.




\section{Background and Preliminaries}
\label{sec:background}

\subsection{Notation}
We denote the set of real numbers by \(\R\).
Let \(Z\) be a square matrix, and we use the notation \(Z^T\), \( \|Z\| \),  \(\tr{Z}\), \(\sigma_{\min}(Z)\) to denote its transpose, spectral norm, trace, and minimum singular value, respectively. 
We indicate positive definite matrices by \(Z\succ 0\).
Given matrices \(\{D_i\}_{i = 1}^m\),  let \( \diag{D_1, \ldots, D_m}\) denote the block diagonal matrix whose $(i,i)$-th block is $D_i$. 
An identity matrix of dimension \(n\) is denoted by \(I_n\).
We use \(\otimes\) to denote the Kronecker product, and \(\vect{X}\) to denote the vectorization of the matrix \(X\) formed by stacking the columns of \(X\) into a single column.
We use \( \hat{e}_i \) to denote the canonical basis vector in \( \R^n \), where the only nonzero entry is the index \(i\).
The normal distribution with mean \(m \in \R^n\) and covariance \(\Lambda \in \R^{n\times n}\) is denoted by \(\Nor(m, \Lambda)\). 

\subsection{Markovian Jump Linear Systems}
\label{sec:MJLSreview}
A Markovian jump linear system is governed by the discrete-time state-space model
\begin{equation} \label{eq:ltv}
x_{t+1} = A_{\omega_t} x_t + B_{\omega_t} u_t+e_t, \text{ with } x_0\sim\mathcal{D} \text{ and } e_t\sim\Nor(0, \eps^2 I),
\end{equation}
where \( x_t \in \R^d\) and \(u_t \in \R^k \) correspond to the state and control action at time \(t\in\mathbb{N}_0\), respectively.
The matrices \( A_{\omega_t} \in \R^{d\times d}\) and \(B_{\omega_t} \in \R^{d\times k} \) depend on a jump parameter $\omega_t$, which is sampled from a Markov chain with a discrete state space \( \Omega:=\{1, \ldots, n_s \} \). 
Hence we have \(A_{\omega_t}\in\{A_i\}_{i\in\Omega}\) and \(B_{\omega_t}\in\{B_i\}_{i\in\Omega}\).
Denote the transition probabilities and initial distribution of $\omega_t$ as \(p_{ij} = \Prob{\omega_{t+1} = j | \omega_t = i}\) and \(\rho = \bmat{\rho_1\! & \cdots & \!\rho_{n_s}}^T\).
We have  \(\sum_{j=1}^{n_s} p_{ij} = 1\), and \( \sum_{i\in\Omega} \rho_i = 1 \).

In this paper, we are interested in minimizing the following discounted quadratic cost
\begin{equation} \label{eq:switched_cost}
C = \Expxt{\sumtinf \gamma^t(x_t^T Q_{\omega_t} x_t + u_t^T R_{\omega_t} u_t)},
\end{equation}
where \(Q_{\omega_t} \succ 0\), \(R_{\omega_t} \succ 0\) and \(\gamma \in (0, 1)\).
When the model information is available, the above MJLS LQR problem can be solved using standard Algebraic Riccati Equation (ARE) techniques~\citep{fragoso}. Specifically, 
let \( \{ P_i \}_{i\in\Omega} \) be the positive definite solution to the following coupled AREs:
\begin{equation}\label{eq:markov_riccati}
P_i = Q_i + \gamma A_i^T \EPnext A_i - \gamma^2 A_i^T \EPnext B_i \left( R_i +\gamma B_i^T \EPnext B_i \right)^{-1} B_i^T \EPnext A_i,
\end{equation}
where $\mathcal{E}_i(P) \coloneqq \Exp{ P_{\omega_{t+1}} \middle| \omega_t = i } = \sum_{j = 1}^{n_s} p_{ij} P_j$.
It is known that the optimal cost can be achieved using a state-feedback controller \(u_t = -K^*_{\omega_t} x_t\) where
\(
K^*_i =\gamma\left( R_i + \gamma B_i^T \EPnext B_i \right)^{-1} B_i^T \EPnext A_i.
\)
In this paper, we are interested in model-free learning of $K_i^*$ for the case where the model parameters $A_i$, $B_i$, $Q_i$, $R_i$, and $p_{ij}$ are unknown.

\subsection{A Brief Review of Policy Learning for LTI Systems}
\label{sec:LQRreview}
Here we briefly review model-free policy learning  for LTI systems. 
LTI systems are a special case of MJLS, where $\Omega = \{1\}$.
When applying policy-based RL methods for LTI systems, one first needs to specify the policy parameterization.
Since the state and action spaces are continuous, it is quite natural to adopt the linear Gaussian policy $u_t\sim \N(-Kx_t,\sigma^2 I)$, where $K$ and $\sigma$ are the parameters to be learned from data.
Then, it is straightforward to apply REINFORCE or natural policy gradient to update $(K,\sigma)$.
In general, it is difficult to obtain finite sample guarantees for REINFORCE and its variants. 
In \cite{pmlr-v80-fazel18a}, it is shown that the population dynamics of the NPG method has linear convergence to the optimal policy if a stabilizing initial policy is used. 
In addition, the authors also present a finite sample analysis of the model-free zeroth-order optimization~\citep{conn2009introduction,nesterov2017random} implementation of the NPG method. 
Notice that zeroth-order optimization (or evolutionary strategies) does not require a stochastic policy for exploration and hence the authors consider a deterministic policy. 
Consequently, their finite sample analysis cannot be directly extended for REINFORCE.
Nevertheless, it is expected that REINFORCE will work for the LTI problem as long as the gradient estimations are reasonably close to the true gradient.

\section{Policy Learning for MJLS}
\label{sec:pg}

Now we recast the MJLS LQR problem within the RL framework. This formulation involves a MDP with mixed continuous/discrete state variables. 
For the MJLS \eqref{eq:ltv}, the system state $\{x_t\}$ itself does not form a Markov chain anymore. 
However, if we augment $(x_t, \omega_t)$ as the new state, we will obtain a hybrid MDP with mixed continuous/discrete state variables. 
Now, the full state space is the product $\R^{n_x}\times \Omega$, but action space is still $\R^{n_u}$. 
The joint state transition model is specified by both the transition probability $\{p_{ij}\}$ and the MJLS model \eqref{eq:ltv}. 

The above MDP adopts a model-based dynamic programming solution which is summarized as the ARE approach reviewed in Section \ref{sec:MJLSreview}. The model-based approach requires knowing the model parameters $(A_i, B_i, Q_i, R_i, p_{ij})$ in advance.
Alternatively, when the model is unknown, one can apply model-free RL algorithms to solve this problem.
This makes the MJLS LQR problem arguably  the most basic benchmark for RL with hybrid MDPs. Understanding the performance of RL algorithms on this benchmark may shed light on how to solve more complicated hybrid MDPs. 
In this paper, we  focus on applying policy-based RL methods for solving such hybrid MDPs.

\subsection{Policy Parameterization and Optimization Landscape}
To apply policy optimization for the above MJLS MDP, we need to confine the search to some certain class of policies. 
Recently the important role of the policy representation has been recognized. 
For hybrid MDPs. there are multiple choices for policy parameterization.
One choice is to adopt a neural network structure where both $x_t$ and $\omega_t$ are fed as inputs. 
However, the optimization landscape for such a neural network parameterization is unclear. 
Another choice is based on optimal control theory for MJLS. 
Since we know the the optimal cost for the MJLS LQR problem can be achieved by a control law in the form of $u_t=-K_{\omega_t}x_t$, it is reasonable to  restrict the policy search within the class of state feedback controllers in the form of $u_t\sim \N(-K_{\omega_t} x_t, \sigma^2 I)$.
Specifically, we can set \(\hat{K} = \bmat{K_1 & \cdots & K_{n_s} }\), where \(K_i\) is the feedback gain for mode \(i\). 
With this notation, we obtain a policy optimization problem whose decision variables are $\hat{K}$ and $\sigma$.

Our policy parameterization can be thought as a mixture of continuous Gaussian policy and discrete look-up tables. 
For each of system modes, we train a corresponding linear policy specified by $K_i$. 
Eventually there are $n_s$ different linear policies stored in a look-up table for various possible values of $\omega_t$. One advantage of our parameterization is that it becomes clear that the cost function $C(\hat K, \sigma)$ only has one stationary point which is the global minimum for the MJLS LQR problem. 
To see this, we first write down an analytical formula for \(\nabla C(\hat{K}, \sigma)\).
Let $P_i^{\hat{K}}$ denote the solution to the coupled Lyapunov equations:
\begin{equation}\label{eq:lyap_markov}
P^{\hat{K}}_i = Q_i + K_i^T R_i K_i + \gamma\left( A_i - B_i K_i \right)^T \EPKnext \left( A_i - B_i K_i \right), \text{ for } i \in \Omega.
\end{equation}
Then the cost~\eqref{eq:switched_cost} subject to the system dynamics ~\eqref{eq:ltv} and the Gaussian policy $u_t\sim \N(-K_{\omega_t} x_t, \sigma^2 I)$ can be calculated as
\begin{equation} \label{eq:markov_cost}
C(\hat K, \sigma) = \Expx{ \sum_{i\in\Omega } \rho_i \left(x_0^T  P_i^{\hat{K}} x_0 + z_i\right)},
\end{equation}
where $z_i$ is solved from the following linear equation\footnote{It is well-known that this equation has a unique solution due to the properties of the transition probability matrix.}:
 \begin{equation} \label{eq:z_i}
 z_i = \sigma^2 \tr{R_i + \gamma B_i^T \EPnext B_i} + \gamma\eps^2 \tr{\EPnext} + \gamma \sum_{j \in \Omega} p_{ij} z_j 
 \end{equation}
Denote \(X_i(t) \coloneqq \Exp{x_t x_t^T \1{\omega_t = i}}\). Notice that \(X_i(t)\) can be recursively determined as
\begin{equation}\label{eq:X_i}
    X_j(t+1)=\sum_{i\in \Omega} p_{ij}\left((A_i-B_i K_i) X_i(t) (A_i-B_i K_i)^T +(\sigma^2 B_i B_i^T +\eps^2 I) \1{\omega_t = i}\right)
\end{equation}
Then we can calculate the policy gradient $\nabla C(\hat{K}, \sigma)$ using the following explicit formula.

\begin{lemma} \label{lemma:policy_grad}
Given $\hat{K}$ stabilizing the scaled system $x_{t+1}=\sqrt{\gamma}(A_{\omega_t}-B_{\omega_t} K_{\omega_t})x_t$ in the mean square sense and $\sigma \geq 0$, the gradient of~\eqref{eq:markov_cost} with respect to control gain \(\hat{K}\) and noise level $\sigma$ is given as
\begin{equation}\label{eq:exact_grad}
\nabla C(\hat{K}, \sigma) = \bmat{\vect{2\bmat{L_1(\hat{K}) & L_2(\hat{K}) & \cdots & L_{n_s}(\hat{K})} \chi_{\hat{K}}} \\ \frac{\sigma}{1-\gamma}\sum_{j = 1}^{n_s} \rho_i \,\, \textup{tr}\big(R_i + \gamma B^T_i \EPKnext B_i\big)} = \bmat{\vect{F_K} \\ F_\sigma}
\end{equation}
where \(L_i(\hat{K}) = \big(R_i + \gamma B^T_i \EPKnext B_i \big) K_i - \gamma B^T_i \EPKnext A_i\), and
\begin{equation} \label{eq:chidef}
\chi_{\hat{K}} =  \sumtinf \gamma^t \textup{diag} \left( X_1(t), \ldots, X_{n_s}(t) \right).
\end{equation}
Later, for simplicity, we use \( \nabla_{\hat{K}}C(\hat{K},\sigma) \coloneqq \nabla C(\hat{K}) = F_K \) and \( \nabla_{\hat{\sigma}}C(\hat{K},\sigma) \coloneqq \nabla C(\sigma) = F_{\sigma} \)\footnote{We slightly abuse our notation here. Notice $F_K$ is a matrix and $\nabla C(\hat{K},\sigma)$ is a vector. When calculating $F_K$, we take the gradient with respect to all entries of $\hat{K}$. Then we augment the columns of $F_K$ with $F_\sigma$ to obtain $\nabla c(\hat{K},\sigma)$.}.
\end{lemma}

\begin{proof}
The differentiability of $ C(\hat K, \sigma)$ can be proved using the implicit function theorem, and this step is similar to the proof of Lemma 3.1 in \citet{rautert1997computational}.
The derivation of the gradient formula follows the similar steps to Lemma 1 in \cite{joaoACC}.  The difference is that now we need to apply \eqref{eq:z_i} and \eqref{eq:X_i}   in our recursive derivations. We can use \eqref{eq:z_i} and \eqref{eq:X_i} to prove the following key fact:
\begin{align*}
&\sum_{i\in\Omega} \tr{dP_i^{\hat{K}} X_i(t)} + \sum_{i\in\Omega } \1{\omega_t = i}dz_i^{\hat{K}}\\ =&\sum_{i\in\Omega}\tr{2dK_i^T L_i(\hat{K}) X_i(t)} 
+ \gamma\left(\sum_{i\in\Omega} \tr{dP_i^{\hat{K}} X_i(t+1)} + \sum_{i\in\Omega } \1{\omega_{t+1} = i}dz_i^{\hat{K}}\right),
\end{align*}
which can be recursively applied to show the gradient formula.
\end{proof}

Suppose \(\Expx{x_0 x_0^T}\) is full rank and $\rho_i>0$ for all $i$.
Then a stationary point given by $\nabla C(\hat{K}, \sigma)=0$  has to satisfy
$L_i(\hat{K}) = \big(R_i + \gamma B^T_i \EPKnext B_i \big) K_i - \gamma B^T_i \EPKnext A_i=0 \text{ for all } i\in\Omega, \text{ and } \sigma = 0$.
It becomes obvious that $L_i(\hat{K}) = 0$ leads to the global optimal policy $\hat{K}^*$ defined by~(\ref{eq:markov_riccati}).
Hence the only stationary point is the global minimum of the original MJLS LQR problem. 

Notice that the  optimization of $C(\hat K, \sigma)$ is a constrained optimization problem whose feasible set consists of all $\hat{K}$ stabilizing  a scaled system $x_{t+1}=\sqrt{\gamma}(A_{\omega_t}-B_{\omega_t} K_{\omega_t})x_t$ in the mean square sense. 
The cost function $C(\hat K, \sigma)$ is finite and differentiable only within the feasible set. Clearly, the size of the feasible set depends on the discounted factor $\gamma$, and the scaling factor $\sqrt{\gamma}$ is standard.

\subsection{Linear Convergence of the Population Dynamics of NPG}

If the model information is known, one can solve $\hat{K}^*$ using the following model-based NPG updates:
\begin{equation}\label{eq:npgd}
\hat{K}^{n+1} = \hat{K}^n-\eta_n \sigma_n^2 \nabla C(\hat{K}^n) \chi_{\hat{K}^n}^{-1}, \text{ and } 
\sigma_{n+1} = \sigma_n - \alpha_n \sigma_n^2 \nabla C(\sigma_n) \left( \frac{1-\gamma}{2 k} \right) 
\end{equation}
The  initial policy is denoted as $\hat{K}^0$ which is assumed to stabilize the $\sqrt{\gamma}$-scaled closed-loop dynamics in the mean square sense.
In \cite{joaoACC}, it is shown that the model-based NPG updates for the MJLS LQR problem with a deterministic policy parameterization and $\gamma=1$\footnote{In \cite{joaoACC}, model-based policy optimization is considered. One does not need a stochastic policy for exploration. Hence it is sufficient to use a deterministic policy there, and the discounted factor is not needed for the problem formulation} can be guaranteed to stay in the feasible set and converge to the global minimum. 
We can obtain similar results for the discounted case with a Gaussian policy, and prove the global convergence of \eqref{eq:npgd}.
In following sections, we will consider the case where the model information is unknown, and implement the model-free natural policy gradient method whose population dynamics exactly matches the above model-based updates. 
We will use model-free RL techniques to estimate $\nabla C(\hat{K})$ and $\chi_{\hat{K}}$ from data. It is expected that the model-free NPG updates will closely track the dynamics of \eqref{eq:npgd} and work well if sufficient data is provided for the gradient estimation.
Connections between NPG and the exact dynamics \eqref{eq:npgd} are further discussed in Appendix \ref{sec:A1}. For completeness, a self-contained proof for the linear convergence of \eqref{eq:npgd} is presented in Appendix \ref{sec:A2}.



\subsection{Natural Policy Gradient and REINFORCE}

Now we discuss the model-free policy learning of the MJLS LQR problem.
From the exact natural policy gradient update rule~\eqref{eq:npgd}, we need to obtain estimates for both the policy gradient and state covariances.
Based on~\eqref{eq:chidef}, we can directly estimate $\chi_{\hat K}$ by averaging $\sum_{t=0}^{T_F} \gamma^t x_t x_t^T$ (with some large $T_F$) over multiple sampled trajectories of the MJLS model \eqref{eq:ltv}.

To estimate $\nabla C(\hat K, \sigma)$ we will adopt the REINFORCE algorithm, which uses a Monte Carlo rollout to estimate the policy gradient. Specifically, for a stochastic policy $\pi_\theta(u_t\vert x_t,w_t)$, we can set 
$\nabla C(\theta) = \mathbb{E} \left[ \sumtinf \gamma^t \nabla_\theta \log \pi_\theta(u_t | x_t, \omega_t) \Psi_t \right]$,
where $\Psi_t$ can be calculated using one of following choices: 1) total reward of the trajectory; 2) reward following input $u_t$; 3) baseline version of the previous item; 4) state-action value function; 5) advantage function; 6) TD residual; 7) generalized advantage estimation.
In this paper, we consider the baseline approach. 
Let us align the notation as
\begin{align*}
\theta=\bmat{\vect{\hat{K}}\\ \sigma}.
\end{align*}
It is known that we can add an arbitrary baseline without affecting the expectation of the policy gradient.
Here we used the cumulative average reward at timestep $t$ as our baseline. 
The policy gradient above is in terms of an expectation, so we can use sampling methods to approximate it as
\begin{equation*}
\nabla C(\theta) \approx \frac{1}{N} \sum_{i = 1}^N \left( \sum_{t=0}^{T_F} \gamma^t \nabla_\theta \log \pi_\theta(u_{t,i} | x_{t,i}, \omega_{t,i}) \left( \sum_{t'=t}^{T_F} \gamma^{t' - t} c_{t',i} - b_{t,i} \right) \right) \eqqcolon \widehat{\nabla C}(\theta)
\end{equation*}
for sufficiently large $N$ and $T_F$. Here, $N$ is the number of sampled trajectories and $T_F$ is the horizon length.
 We calculate the baseline as \(b_{t,j} = \frac{1}{j}\sum_{i = 1}^j \left( \sum_{t'=t}^{T_F} \gamma^{t' - t}c_{t',i} \right)\).
We use a cumulative formulation of the baseline so it is not necessary to store all previous trajectories.
Algorithm~\ref{alg:alg1} below provides a procedure to compute the estimates \(\widehat{\nabla C}(\hat{K})\) and \(\widehat{\chi}_{\hat{K}^{n}}\)

Having obtained a model-free policy gradient estimation, we can directly use the gradient estimates to update the controller gains with the natural policy gradient step:
\begin{equation}\label{eq:mf_npgd}
\hat{K}^{n+1} =  \hat{K}^{n} - \eta_n \sigma_n^2 \widehat{\nabla C}(\hat{K}^{n}) \widehat{\chi}_{\hat{K}^{n}}^{-1}, \text{ and }
\sigma_{n+1} = \sigma_n - \alpha_n \sigma_n^2 \widehat{\nabla C}(\sigma_n) \left( \frac{1-\gamma}{2 k} \right)
\end{equation}

We now discuss Algorithm~\ref{alg:alg1}, which is repeated at every NPG iteration step.
To run Algorithm~\ref{alg:alg1}, we need to specify the control policy \(\hat{K}\), an input variance \(\sigma^2\), the single trajectory length \(T_F\), and the total number of trajectories \(N\).
For each trajectory generated we need to measure the systems states \(x_t\), actions \(u_t\), state-action cost \(c_t\), and the jump parameter \(\omega_t\).
We note that it is common to use a constant exploration noise $\sigma$ in many RL algorithms, however here we are treating $\sigma$ as a parameter one wants to learn. 
While we know that the optimal $\sigma$ for the LQR problem is zero, having a variable $\sigma$ helps finding a better trade-off between exploration and exploitation. 
For our simulations, we set up $\alpha_n$ in a way that $\sigma$ decreases to $0$ at a pre-specified linear rate $0.99$.

\begin{algorithm2e}
 Starting from a control policy \(\hat{K}\), input variance \(\sigma^2\), trajectory length \(T_F\), and batch size \(N\)\;
 \For{$i = 1,\ldots, N$}{
  Generate a trajectory $\tau_i$ and measure $\{ x_t, u_t, \omega_t, c_t \}_{t = 0, \ldots, T_F}$ starting from $x_0 \sim \mathcal{D}$ with $u_t \sim \mathcal{N}(-K_{\omega_t} x_t, \sigma^2 I)$\;
  Initialize $b_t = 0$ and $\hat{X}_{j, i} = 0$ for all $j\in\Omega$ and $t\in[0, T_F]$\;
    \For{$t = 0,\ldots, T_F$}{
      Compute:\\
      $\hat{G}_{t} = -\frac{1}{\sigma^2}  (u_t + K_{\omega_t} x_t) \left( \hat{e}_{\omega_t} \otimes x_t \right)^T$,\\
      $\hat{Q}_t = \sum_{t' = t}^{T_F} \gamma^{t' - t} c_t$,\\
      $\hat{S}_t = -\frac{k}{\sigma} + \frac{1}{\sigma^3}(u_t + K_{\omega_t} x_t)^T (u_t + K_{\omega_t} x_t)$,\\
       \( \hat{X}_{\omega_t, i} = \hat{X}_{\omega_t, i} + \gamma^t x_t x_t^T  \), and\\
       \( b_t = ((i-1)b_t + \hat{Q}_t)/i \)\;
    }
    Compute \(\widehat{\nabla C}_i (\hat{K}) = \sum_{t = 0}^{T_F} \gamma^t \hat{G}_{t} (\hat{Q}_t - b_t) \) and \(\widehat{\nabla C}_i (\sigma) = \sum_{t = 0}^{T_F} \gamma^t \hat{S}_{t} (\hat{Q}_t - b_t) \)\;
 }
 Return the estimates: \(\widehat{\nabla C}(\cdot) = \frac{1}{N}\sum_{i = 0}^N \widehat{\nabla C}_i(\cdot)\)  and \(\widehat{\chi}_{\hat{K}^{n}} = \frac{1}{N}\sum_{i = 0}^N \text{diag}\big(\hat{X}_{1, i}, \ldots, \hat{X}_{n_s, i}\big)\) 
 \caption{Model-Free Switched Policy Gradient Estimation}
 \label{alg:alg1}
\end{algorithm2e}


\section{Model-Free Implementations}
\label{sec:DataNPG}

In this section, we implement the model-free policy gradient algorithm to different example systems.

\subsection{Small Scale Problem}
Consider a MJLS which can switch between two modes, where each mode \((A_1, B_1)\) and \((A_2, B_2)\) are not individually stabilizable, but the switched system is. We define the state space matrices as:
\begin{equation*}
A_1 = \bmat{0.4 & 0.6 & -0.1 \\ -0.4 & -0.6 & 0.3 \\ 0 & 0 & 1}\!, \quad B_1 = \bmat{1 \\ 1 \\ 0}\!, \quad
A_2 = \bmat{0.9 & 0.5 & -0.1 \\ 0 & 1 & 0 \\ -0.1 & 0.5 & -0.4}\!, \quad B_2 = \bmat{1 \\ 0 \\ 1}\!,
\end{equation*}
with cost matrices \(Q_1 = I_3\), \(R_1 = 1\), \(Q_2 = 2 I_3\), and \(R_2 = 2\).
We also set the transition probability \(\mathcal{P}= \bmat{0.7 & 0.3 \\ 0.4 & 0.6}\), and initial distribution \(\rho = \bmat{0.5 \\ 0.5}\).

Using Algorithm~\ref{alg:alg1} with \eqref{eq:mf_npgd}, we computed the policy of the system above.
For all simulations, we set \(K_1 = K_2 = 0_{k\times d}\) as our initial gain values since \(C(\hat{K}^0, \sigma_0)\) is finite, and used the parameters \(T = 500\), \(\sigma_0 = 0.5\), \( \gamma =  0.99 \), and \(\eta = 0.01\sigma_0^{-2}\). 
For \( N \in \{1000, 2500, 5000, 10000\}\), we ran 100 steps of~\eqref{eq:mf_npgd} and computed the associated costs of the policy obtained.
Each policy iteration was computed 1000 times, with the results shown in Figure~\ref{fig:small}.

\subsection{System with Large Number of Modes}

We now consider a system with 100 states, 20 inputs, and 100 modes.
The matrices \(A\) and \(B\) were generated using \texttt{drss} in MATLAB in order to guarantee that the system would have finite cost with \(\hat{K}^0 = 0\).
The probability transition matrix \(\mathcal{P}\) was sampled from a Dirichlet Process \( \text{Dir}(99\cdot I_{100} + 1) \), which always results in an irreducible Markov chain.
For simplicity, we set \(\rho_i = 1/100\), \( Q_i = I \), and \( R_i = I \) for all \( i \in \Omega \).
Here, we used the step size \(\eta = 0.000125\sigma_0^{-2} \), initial noise \( \sigma_0 = 1 \), and batch sizes \(N = \{25000, 50000\}\).
The resulting policy costs are shown in Figure~\ref{fig:large}.

Obtaining controllers for systems with a large number of modes can be computationally hard using \eqref{eq:markov_riccati}, as the number of coupled equations grows with the number of modes.
By using the data-driven approach, we only need to assert that we visit each mode often enough.
Here, this condition is directly satisfied since the modes are sampled from an irreducible Markov chain.

\begin{figure}
\centering
\floatconts
  {fig:cost_diff}
  {\vspace{-22pt}\caption{Shown are the relative error between policy \( \hat{K}^i \), obtained using Algorithm~\ref{alg:alg1}, and the optimal policy \(\hat{K}^*\), computed using~(\ref{eq:markov_riccati}).
  The relative error was calculated as \( \left|\frac{C(\hat{K}^i, 0) - C(\hat{K}^*, 0)}{C(\hat{K}^*, 0)}\right| \times 100  \).
  The solid lines indicate the mean expected percent error.}\vspace{-10pt}}
  {%
    \subfigure[Small scale example]{%
      \label{fig:small}%
      \includegraphics[width=0.475\textwidth]{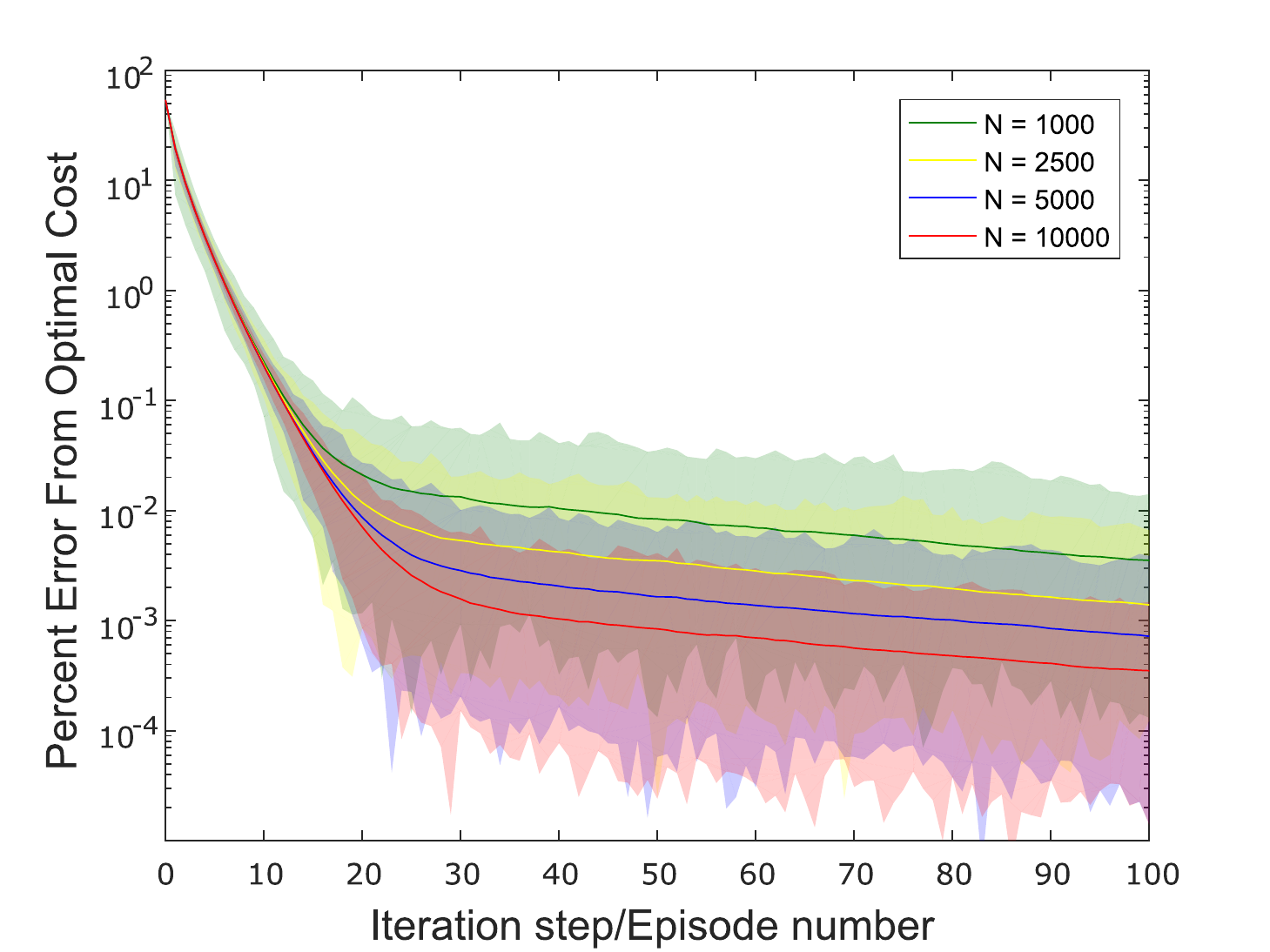}
    }\quad 
    \subfigure[Large number of modes]{%
      \label{fig:large}%
      \includegraphics[width=0.475\textwidth]{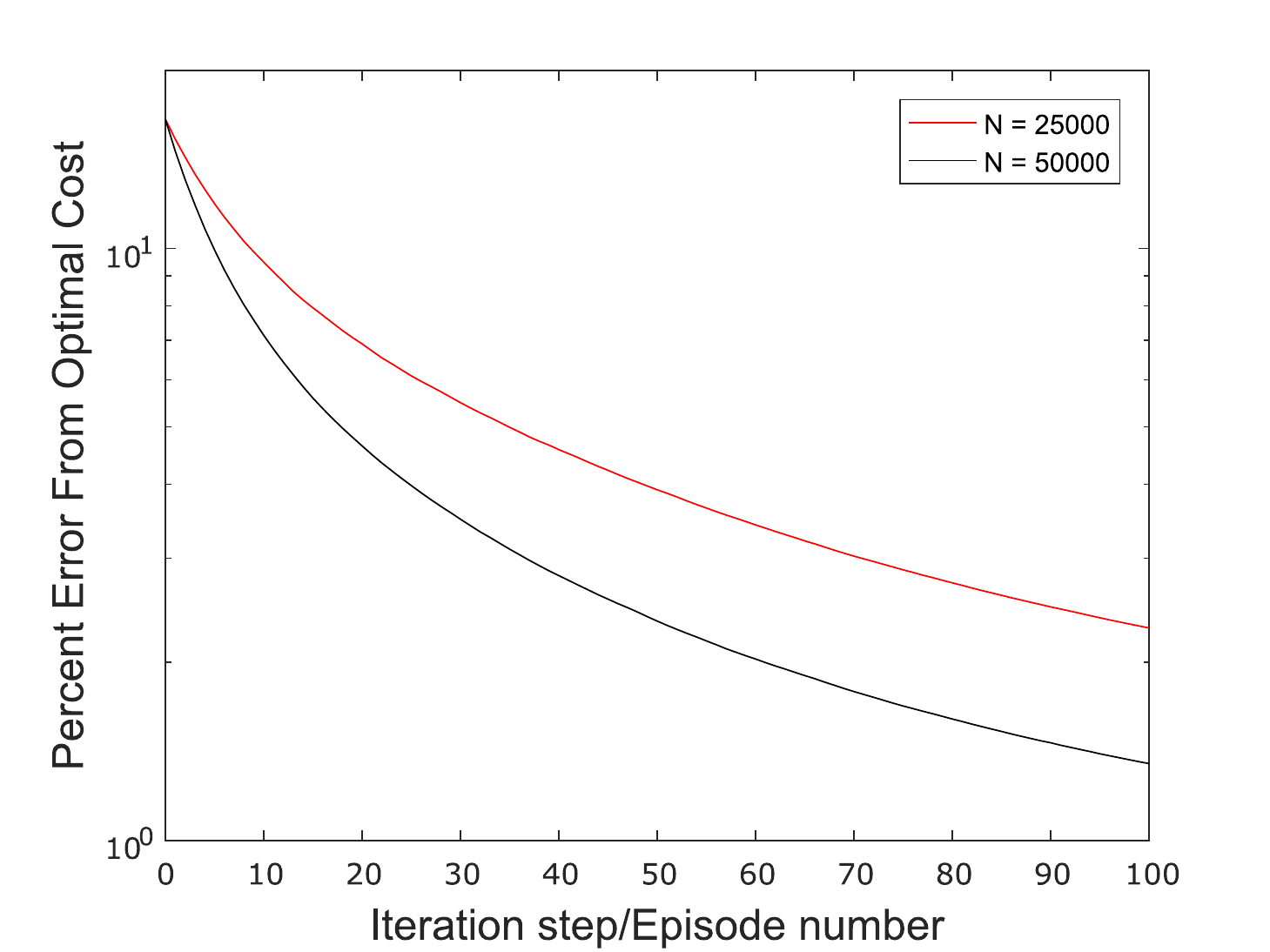}
    }
  }
  
\end{figure}

\subsection{Structured Controller}

Now we consider the case where we want to impose some certain structure to the designed controller.
For example, in the output feedback problem, the controller cannot access the full state measurements.
Structured control design also finds many applications in decentralized control, where individual controllers might not have access to the global system state.
To find the optimal structured controller, we can simply project the estimated gradient to maintain the desired structure, and update the control gains using gradient descent, instead of the natural policy gradient.

To illustrate how the projected gradient descent works for the structured control design problem, we test the algorithm on a small example system.
Consider the following two-mode system:
\begin{equation*}
A_1 = \bmat{-0.4 & 1.0 \\ 0.0 & 0.9}\!, \quad
A_2 = \bmat{0.0 & 1.0 \\ -0.4 & 0.9}\!, \quad \{B_i\}_{i = 1,2} = \bmat{1.0 & 0.5\\ 0.0 & 2.0}\!,
\end{equation*}
with the weighting matrices and transition probability:
\begin{equation*}
\{Q_i\}_{i = 1,2} =\bmat{10 & 0 \\ 0 & 20}\!, \quad \{R_i\}_{i = 1,2} = \bmat{1 & 0 \\ 0 & 1}\!, \quad \mathcal{P} = \bmat{0.8 & 0.2 \\ 0.3 & 0.7}, \quad \rho = \bmat{0.5 \\ 0.5}.
\end{equation*}
Using \eqref{eq:markov_riccati} and \eqref{eq:lyap_markov}, the total expected cost following the optimal unstructured policy, $\hat{K}^*_{unstruc}$, is $C(\hat{K}^*_{unstruc}) = 2.5704$, while the expected cost of having no feedback, $\hat{K}_0$, is $C(\hat{K}_0) = 8.4861$.

Now suppose that we can only measure the first state of the system.
This is equivalent to having a controller of the form
\begin{equation}\label{eq:K_shape}
K_{shape} = \bmat{\bullet & 0 \\ \bullet & 0}.
\end{equation}
 If we project the estimate of the gradient onto \eqref{eq:K_shape}, and iterate using gradient descent, we obtain the expected cost $C(\hat{K}^{100}_{struct}) = 6.2226$ after 100 iteration steps. Clearly, the resultant structured control gain is not simply the projection of the optimal unstructured control gain onto the structured space.
If we simply project $\hat{K}^*_{unstruc}$ onto \eqref{eq:K_shape}, and denote that by $\hat{K}^*_{proj}$, the total expected cost becomes $C(\hat{K}^*_{proj}) = 13.3227$, which is worse than not having any feedback action at all.

\section{Conclusion}
In this paper we revisited the optimal control of Markovian Jump Linear Systems as a benchmark for further understanding of policy-based RL algorithms and hybrid MDPs.
We discussed how to set up the policy parameterization for such hybrid MDPs, and present an efficient data-driven implementation of the natural policy gradient method for learning optimal state-feedback controllers of unknown MJLSs. 
We demonstrated the performance of the model-free natural policy gradient method on different example systems.
Our results suggest that it is promosing to apply policy-based RL methods for optimal control of large scale switching systems, where the computational complexity grows as the system size increases.

\bibliography{nips}

\begin{thebibliography}{41}
\providecommand{\natexlab}[1]{#1}
\providecommand{\url}[1]{\texttt{#1}}
\expandafter\ifx\csname urlstyle\endcsname\relax
  \providecommand{\doi}[1]{doi: #1}\else
  \providecommand{\doi}{doi: \begingroup \urlstyle{rm}\Url}\fi

\bibitem[Abbasi-Yadkori and Szepesv{\'a}ri(2011)]{abbasi2011regret}
Y.~Abbasi-Yadkori and C.~Szepesv{\'a}ri.
\newblock Regret bounds for the adaptive control of linear quadratic systems.
\newblock In \emph{Proceedings of the 24th Annual Conference on Learning
  Theory}, pages 1--26, 2011.

\bibitem[Abbasi-Yadkori et~al.(2018)Abbasi-Yadkori, Lazic, and
  Szepesv{\'a}ri]{abbasi2018regret}
Y.~Abbasi-Yadkori, N.~Lazic, and C.~Szepesv{\'a}ri.
\newblock Regret bounds for model-free linear quadratic control.
\newblock \emph{arXiv preprint arXiv:1804.06021}, 2018.

\bibitem[Bar-Shalom and Li(1993)]{bar1993estimation}
Y.~Bar-Shalom and X.~Li.
\newblock Estimation and tracking- principles, techniques, and software.
\newblock \emph{Norwood, MA: Artech House, Inc, 1993.}, 1993.

\bibitem[Beirigo et~al.(2018)Beirigo, Todorov, and Barreto]{beirigo2018online}
Rafael~L Beirigo, Marcos~Garcia Todorov, and Andr{\'e} da Motta~Salles Barreto.
\newblock Online {TD} ($\lambda$) for discrete-time {M}arkov jump linear
  systems.
\newblock In \emph{2018 IEEE Conference on Decision and Control (CDC)}, pages
  2229--2234, 2018.

\bibitem[Boyan and Littman(2001)]{boyan2001exact}
J.~A Boyan and M.~L. Littman.
\newblock Exact solutions to time-dependent {MDP}s.
\newblock In \emph{Advances in Neural Information Processing Systems}, pages
  1026--1032, 2001.

\bibitem[Conn et~al.(2009)Conn, Scheinberg, and Vicente]{conn2009introduction}
A.~Conn, K.~Scheinberg, and L.~Vicente.
\newblock \emph{Introduction to derivative-free optimization}, volume~8.
\newblock Siam, 2009.

\bibitem[Costa et~al.(2006)Costa, Fragoso, and Marques]{costa2006discrete}
O.~Costa, M.~Fragoso, and R.~Marques.
\newblock \emph{Discrete-time {M}arkov jump linear systems}.
\newblock Springer London, 2006.

\bibitem[Costa and Aya(2002)]{costa2002monte}
Oswaldo~LV Costa and Julio~CC Aya.
\newblock Monte {C}arlo {TD} ($\lambda$)-methods for the optimal control of
  discrete-time {M}arkovian jump linear systems.
\newblock \emph{Automatica}, 38\penalty0 (2):\penalty0 217--225, 2002.

\bibitem[Dean et~al.(2017)Dean, Mania, Matni, Recht, and Tu]{dean2017sample}
S.~Dean, H.~Mania, N.~Matni, B.~Recht, and S.~Tu.
\newblock On the sample complexity of the linear quadratic regulator.
\newblock \emph{arXiv preprint arXiv:1710.01688}, 2017.

\bibitem[Dean et~al.(2018)Dean, Mania, Matni, Recht, and Tu]{dean2018regret}
S.~Dean, H.~Mania, N.~Matni, B.~Recht, and S.~Tu.
\newblock Regret bounds for robust adaptive control of the linear quadratic
  regulator.
\newblock In \emph{Advances in Neural Information Processing Systems}, pages
  4188--4197, 2018.

\bibitem[Duan et~al.(2016)Duan, Chen, Houthooft, Schulman, and
  Abbeel]{duan2016benchmarking}
Y.~Duan, X.~Chen, R.~Houthooft, J.~Schulman, and P.~Abbeel.
\newblock Benchmarking deep reinforcement learning for continuous control.
\newblock In \emph{International Conference on Machine Learning}, pages
  1329--1338, 2016.

\bibitem[Fazel et~al.(2018)Fazel, Ge, Kakade, and Mesbahi]{pmlr-v80-fazel18a}
M.~Fazel, R.~Ge, S.~Kakade, and M.~Mesbahi.
\newblock Global convergence of policy gradient methods for the linear
  quadratic regulator.
\newblock In \emph{Proceedings of the 35th International Conference on Machine
  Learning}, volume~80, pages 1467--1476, 2018.

\bibitem[Fox et~al.(2011)Fox, Jordan, and Willsky]{fox2011tsp}
E.~Fox, E.B. Sudderthand~M.I. Jordan, and A.S. Willsky.
\newblock Bayesian nonparametric inference of switching dynamic linear models.
\newblock \emph{IEEE Transactions on Signal Processing}, 59\penalty0
  (4):\penalty0 1569 -- 1585, 2011.

\bibitem[Fragoso(1989)]{fragoso}
M.~Fragoso.
\newblock Discrete-time jump {LQG} problem.
\newblock \emph{International Journal of Systems Science}, 20\penalty0
  (12):\penalty0 2539--2545, 1989.

\bibitem[Gopalakrishnan et~al.(2016)Gopalakrishnan, Balakrishnan, and
  Jordan]{hamsa2016cdc}
K.~Gopalakrishnan, H.~Balakrishnan, and R.~Jordan.
\newblock Stability of networked systems with switching topologies.
\newblock In \emph{IEEE Conference on Decision and Control}, pages 1889--1897,
  2016.

\bibitem[Guestrin et~al.(2004)Guestrin, Hauskrecht, and
  Kveton]{guestrin2004solving}
C.~Guestrin, M.~Hauskrecht, and B.~Kveton.
\newblock Solving factored {MDP}s with continuous and discrete variables.
\newblock In \emph{Proceedings of the 20th conference on Uncertainty in
  Artificial Intelligence}, pages 235--242, 2004.

\bibitem[He et~al.(2019)He, Zhang, Fang, Liu, Luan, and
  Ding]{he2019reinforcement}
Shuping He, Maoguang Zhang, Haiyang Fang, Fei Liu, Xiaoli Luan, and Zhengtao
  Ding.
\newblock Reinforcement learning and adaptive optimization of a class of
  {M}arkov jump systems with completely unknown dynamic information.
\newblock \emph{Neural Computing and Applications}, pages 1--10, 2019.

\bibitem[Henderson et~al.(2018)Henderson, Islam, Bachman, Pineau, Precup, and
  Meger]{henderson2018deep}
P.~Henderson, R.~Islam, P.~Bachman, J.~Pineau, D.~Precup, and D.~Meger.
\newblock Deep reinforcement learning that matters.
\newblock In \emph{Thirty-Second AAAI Conference on Artificial Intelligence},
  2018.

\bibitem[Hu and Syed(2019)]{hu2019characterizing}
B.~Hu and U.~Syed.
\newblock Characterizing the exact behaviors of temporal difference learning
  algorithms using {M}arkov jump linear system theory.
\newblock In \emph{Advances in Neural Information Processing Systems}, pages
  8477--8488, 2019.

\bibitem[Hu et~al.(2017)Hu, Seiler, and Rantzer]{hu2017unified}
B.~Hu, P.~Seiler, and A.~Rantzer.
\newblock A unified analysis of stochastic optimization methods using jump
  system theory and quadratic constraints.
\newblock In \emph{Conference on Learning Theory}, pages 1157--1189, 2017.

\bibitem[Jansch-Porto et~al.(2020)Jansch-Porto, Hu, and Dullerud]{joaoACC}
J.P. Jansch-Porto, B.~Hu, and G.E. Dullerud.
\newblock Convergence guarantees of policy optimization methods for markovian
  jump linear systems.
\newblock \emph{accepted to ACC}, 2020.

\bibitem[Kakade(2002)]{kakade2002natural}
S.~Kakade.
\newblock A natural policy gradient.
\newblock In \emph{Advances in neural information processing systems}, pages
  1531--1538, 2002.

\bibitem[Levine et~al.(2016)Levine, Finn, Darrell, and Abbeel]{levine2016end}
S.~Levine, C.~Finn, T.~Darrell, and P.~Abbeel.
\newblock End-to-end training of deep visuomotor policies.
\newblock \emph{The Journal of Machine Learning Research}, 17\penalty0
  (1):\penalty0 1334--1373, 2016.

\bibitem[Malik et~al.(2018)Malik, Pananjady, Bhatia, Khamaru, Bartlett, and
  Wainwright]{malik2018derivative}
D.~Malik, A.~Pananjady, K.~Bhatia, K.~Khamaru, P.~Bartlett, and M.~Wainwright.
\newblock Derivative-free methods for policy optimization: Guarantees for
  linear quadratic systems.
\newblock \emph{arXiv preprint arXiv:1812.08305}, 2018.

\bibitem[Nesterov and Spokoiny(2017)]{nesterov2017random}
Y.~Nesterov and V.~Spokoiny.
\newblock Random gradient-free minimization of convex functions.
\newblock \emph{Foundations of Computational Mathematics}, 17\penalty0
  (2):\penalty0 527--566, 2017.

\bibitem[Pavlovic et~al.(2000)Pavlovic, Rehg, and
  MacCormick]{Pavlovic2000LearningSL}
V.~Pavlovic, J.M. Rehg, and J.~MacCormick.
\newblock Learning switching linear models of human motion.
\newblock In \emph{Advances in Neural Information Processing Systems}, 2000.

\bibitem[Peters and Schaal(2008)]{peters2008natural}
J.~Peters and S.~Schaal.
\newblock Natural actor-critic.
\newblock \emph{Neurocomputing}, 71\penalty0 (7-9):\penalty0 1180--1190, 2008.

\bibitem[Rajeswaran et~al.(2017)Rajeswaran, Lowrey, Todorov, and
  Kakade]{rajeswaran2017towards}
A.~Rajeswaran, K.~Lowrey, E.~Todorov, and S.~Kakade.
\newblock Towards generalization and simplicity in continuous control.
\newblock In \emph{Advances in Neural Information Processing Systems}, pages
  6550--6561, 2017.

\bibitem[Rautert and Sachs(1997)]{rautert1997computational}
T.~Rautert and E.~Sachs.
\newblock Computational design of optimal output feedback controllers.
\newblock \emph{SIAM Journal on Optimization}, 7\penalty0 (3):\penalty0
  837--852, 1997.

\bibitem[Schulman et~al.(2015{\natexlab{a}})Schulman, Levine, Abbeel, Jordan,
  and Moritz]{schulman2015trust}
J.~Schulman, S.~Levine, P.~Abbeel, M.~Jordan, and P.~Moritz.
\newblock Trust region policy optimization.
\newblock In \emph{International Conference on Machine Learning}, pages
  1889--1897, 2015{\natexlab{a}}.

\bibitem[Schulman et~al.(2015{\natexlab{b}})Schulman, Moritz, Levine, Jordan,
  and Abbeel]{schulman2015high}
J.~Schulman, P.~Moritz, S.~Levine, M.~Jordan, and P.~Abbeel.
\newblock High-dimensional continuous control using generalized advantage
  estimation.
\newblock In \emph{International Conference on Learning Representation},
  2015{\natexlab{b}}.

\bibitem[Schulman et~al.(2017)Schulman, Wolski, Dhariwal, Radford, and
  Klimov]{schulman2017proximal}
J.~Schulman, F.~Wolski, P.~Dhariwal, A.~Radford, and O.~Klimov.
\newblock Proximal policy optimization algorithms.
\newblock \emph{arXiv preprint arXiv:1707.06347}, 2017.

\bibitem[Sutton and Barto(2018)]{sutton2018reinforcement}
R.~Sutton and A.~Barto.
\newblock \emph{Reinforcement learning: An introduction}.
\newblock MIT press, 2018.

\bibitem[Sutton et~al.(2000)Sutton, McAllester, Singh, and
  Mansour]{sutton2000policy}
R.~Sutton, D.~McAllester, S.~Singh, and Y.~Mansour.
\newblock Policy gradient methods for reinforcement learning with function
  approximation.
\newblock In \emph{Advances in neural information processing systems}, pages
  1057--1063, 2000.

\bibitem[Sworder and Boyd(1999)]{sworder1999estimation}
D.~Sworder and J.~Boyd.
\newblock \emph{Estimation problems in hybrid systems}.
\newblock Cambridge University Press, 1999.

\bibitem[Toussaint and Storkey(2006)]{toussaint2006probabilistic}
M.~Toussaint and A.~Storkey.
\newblock Probabilistic inference for solving discrete and continuous state
  markov decision processes.
\newblock In \emph{Proceedings of the 23rd international conference on Machine
  learning}, pages 945--952. ACM, 2006.

\bibitem[Tu and Recht(2018)]{tu2018gap}
S.~Tu and B.~Recht.
\newblock The gap between model-based and model-free methods on the linear
  quadratic regulator: An asymptotic viewpoint.
\newblock \emph{arXiv preprint arXiv:1812.03565}, 2018.

\bibitem[Vargas et~al.(2013)Vargas, Costa, and do~Val]{varga2013ijrnc}
A.~N. Vargas, E.~F. Costa, and J.~B.~R. do~Val.
\newblock On the control of {M}arkov jump linear systems with no mode
  observation: {A}pplication to a {DC} motor device.
\newblock \emph{International Journal of Robust and Nonlinear Control},
  23\penalty0 (10):\penalty0 1136--1150, 2013.

\bibitem[Williams(1992)]{williams1992simple}
R.~Williams.
\newblock Simple statistical gradient-following algorithms for connectionist
  reinforcement learning.
\newblock \emph{Machine learning}, 8\penalty0 (3-4):\penalty0 229--256, 1992.

\bibitem[Yang et~al.(2019)Yang, Chen, Hong, and Wang]{yang2019global}
Z.~Yang, Y.~Chen, M.~Hong, and Z.~Wang.
\newblock On the global convergence of actor-critic: A case for linear
  quadratic regulator with ergodic cost, 2019.

\bibitem[Zhang et~al.(2019)Zhang, Hu, and Ba{\c{s}}ar]{zhang2019policyc}
Kaiqing Zhang, Bin Hu, and Tamer Ba{\c{s}}ar.
\newblock Policy optimization for $\mathcal{H}_2$ linear control with
  $\mathcal{H}_\infty$ robustness guarantee: Implicit regularization and global
  convergence.
\newblock \emph{arXiv preprint arXiv:1910.09496}, 2019.

\end{thebibliography}

\clearpage
\appendix
\begin{center}
{\Large\bf Supplementary Material}
\end{center}

\section{More Discussions on  Natural Policy Gradient for MJLS}
\subsection{Connections between NPG and the Exact Dynamics \eqref{eq:npgd}}
\label{sec:A1}

Suppose $\theta$ is a column vector parameterizing the controller.
In our case, we have $\theta\coloneqq \vect{\hat{K}}$ (recall that $\hat{K}$ is obtained by horizontally concatenating the control gains for each mode, as \(\hat{K} = \bmat{K_1 & \cdots & K_{n_s} }\)). 
Notice here we fix the noise level $\sigma$ as a constant to simplify the derivations.
From \cite{kakade2002natural}, the natural policy gradient method iterates as follows:
\begin{equation}\label{eq:update}
\theta \leftarrow \theta - \eta G_\theta^{-1} \nabla C(\theta)
\end{equation}
where $G_\theta$ is the Fisher information matrix, which is calculated as
\begin{equation}
G_\theta = \Expx{\sumtinf \gamma^t \left(\nabla \log \pi_\theta (u_t | x_t, \omega_t)\right) \left(\nabla \log \pi_\theta (u_t | x_t, \omega_t) \right)^T}
\end{equation}
Since Gaussian policies satisfy regularity conditions, each element of the Fisher information matrix can also be computed as
\begin{equation}
[G_\theta]_{i,j} = -\Expx{\sumtinf \gamma^t \frac{\partial}{\partial \theta_i \partial \theta_j} \log \pi_\theta (u_t | x_t, \omega_t)}
\end{equation}

Let $f(i) = (i-1) \mod k + 1$, $g(i) = (\lceil i/k \rceil -1) \mod d + 1$, and $h(i) = \lceil t/kd \rceil$.
Then, each element of the Fisher information matrix is
\begin{equation*}
[G_\theta]_{i,j} = \frac{1}{\sigma^2} \Expx{\sumtinf \gamma^t x_{g(i)}(t) x_{g(j)}(t) \1{h(i)=\omega_t} \1{h(j) = \omega_t} \1{f(i) = f(j)} }
\end{equation*}
This is equivalent to
\begin{equation}
G_\theta = \frac{1}{\sigma^2} \left(\chi_{\hat{K}} \otimes I_k\right).
\end{equation}
Recall that for some matrices $A$,$B$,and $X$, we have the relationship $(B^T \otimes A) \vect{X} = \vect{AXB}$. Then the update~\eqref{eq:update} can be written as
\begin{align*}
G_\theta^{-1} \vect{\nabla C(\hat{K})} &= \sigma^2 \left( \chi_{\hat{K}}^{-1} \otimes I_n\right) \vect{\nabla C(\hat{K})} \\
    &= \vect{\sigma^2 I_n \nabla C(\hat{K}) \chi_{\hat{K}}^{-1}} \\
    &= \vect{\sigma^2 \nabla C(\hat{K}) \chi_{\hat{K}}^{-1}}.
\end{align*}
This gives us the update step as in equation~\eqref{eq:npgd}.

\subsection{Convergence of the Exact Dynamics \eqref{eq:npgd}}
\label{sec:A2}

The convergence analysis for \eqref{eq:npgd} is very similar to the deterministic MJLS case \citep{joaoACC}. For completeness, we present a proof here.  
To state the convergence result, it is helpful to define \(\hat{R} = \diag{R_1, \ldots, R_{n_s}}\), and  \(\hat{B} = \diag{B_1 , \ldots , B_{n_s}}\).
We also denote $\mu \coloneqq \min_{i\in \Omega}(\rho_i)\, \sigma_{\min}\!\left( \Expx{x_0 x_0^T} \right)$ and assume $\mu>0$.  
This indicates that there is a chance of starting from any Markov state, and that the expected covariance of the initial states is full rank. For simplicity, we denote the feasible set of this constrained optimization problem as $\mathcal{K}$, i.e. $\mathcal{K}$ consists of all $\hat{K}$ stabilizing the scaled system $x_{t+1}=\sqrt{\gamma}(A_{\omega_t}-B_{\omega_t} K_{\omega_t})x_t$ in the mean square sense. 
Notice $\chi_{\hat{K}}$ depends on $\sigma$, and hence in this section we will denote this term as $\chi_{\hat{K}}(\sigma)$ to make such a dependence explicit. The main convergence result is stated as follows.
\begin{theorem}\label{thm:conv}
Suppose \(\hat{K}^0\in \mathcal{K}\) and \(\gamma \in (0,1)\). 
Denote $\tilde{\eta}:=\eta_n \sigma_n^2$. 
If we fix $\tilde{\eta}$ as a constant satisfying the following bound:
\begin{equation}
    \tilde{\eta} \le \frac{1}{2}\left( \| \hat{R} \| + \gamma \frac{\| \hat{B} \|^2 C(\hat{K}^0, 0)}{\mu} \right)^{-1},
\end{equation}{}
then the  natural policy gradient method~\eqref{eq:npgd} converges to the global minimum $\hat{K}^*$ linearly as follows
\begin{equation}
\label{eq:mainCon}
C(\hat{K}^N, 0) - C(\hat{K}^*, 0) \le \left( 1 - 2\tilde{\eta} \mu  \frac{\sigma_{\min}(\hat{R})}{\| \chi_{\hat{K}^*}(0) \|} \right)^N \left(C(\hat{K}^0, 0) - C(\hat{K}^*, 0) \right).
\end{equation}
\end{theorem}

Before proceeding to the proof of the above result, we make a few remarks here.
The main difference between the above result and Theorem 2 in \citet{joaoACC} is that now we use a discounted cost and an exploration noise is also involved.
Consequently, we need to treat $\eta_n \sigma_n^2$ as a ``scaled" stepsize $\tilde{\eta}$ and fix this quantity as a constant to ensure the linear convergence. If we replace the stepsize in Theorem 2 of \citet{joaoACC} with the scaled stepsize $\tilde{\eta}$, then we obtain the bound \eqref{eq:mainCon}.
 Notice that there is a gap between the above theoretical result and the simulation study presented in the main paper. In the model-free implementations, the natural gradient is noisy and hence one needs to gradually decrease the scaled stepsize $\tilde{\eta}$ for the control of variance. Therefore, in our simulations, we fix $\eta_n$ as a constant and gradually decrease $\sigma_n$. In contrast, for the population dynamics, fixing the scaled stepsize $\tilde{\eta}$  as a constant leads to the linear convergence result.

One may think that it is more reasonable to consider a convergence bound in the form of
\(C(\hat{K}^N, \sigma_N) - C(\hat{K}^*, 0) \le \rho^N \left(C(\hat{K}^0, \sigma_0) - C(\hat{K}^*, 0) \right)\).
Actually that is not the case for the analysis of the population dynamics of the NPG method.
Notice that the evaluation of the exact natural gradient $\nabla_{\hat{K}} C(\hat{K},\sigma)\chi_{\hat{K}}(\sigma)^{-1}$ does not depend on $\sigma$. This is different from the policy gradient case where $\nabla_{\hat{K}} C(\hat{K},\sigma)$ does depend on $\sigma$. 
Therefore, with a fixed scaled stepsize $\tilde{\eta}$, the control gain iterations in the exact NPG do not depend on the injected noise, and it makes sense to consider a bound which is free of the parameter $\sigma$.
We set $\sigma=0$ in the bound \eqref{eq:mainCon} since the optimal value of $\sigma$ is known to be $0$ and the noise is only injected for the exploration purpose.
We want to emphasize that for the population dynamics of the NPG method, the control gain iterations and the policy evaluations can be decoupled such that $\sigma=0$ is only used in evaluating the performance of $\hat{K}^N$ and does not affect the scaled stepsize $\tilde{\eta}$ which is fixed beforehand.

Now we are ready to present the main proof which aligns closely with the proof of Theorem 2 in \cite{joaoACC}. The proof includes three steps.
\begin{enumerate}
    \item Show that the one-step progress of the natural policy gradient gives a policy still in $\mathcal{K}$.
    \item Apply the almost-smoothness condition and the gradient domination property to show that the cost associated by the one-step progress of the natural policy gradient method decreases as follows
\begin{equation*}
C(\hat{K}^{n+1}, 0) - C(\hat{K}^*, 0) \le \left( 1 - 2\tilde{\eta} \mu  \frac{\sigma_{\min}(\hat{R})}{\| \chi_{\hat{K}^*}(0) \|} \right) \left(C(\hat{K}^n, 0) - C(\hat{K}^*, 0) \right).
\end{equation*}
    \item Use induction to prove the bound \eqref{eq:mainCon}.
\end{enumerate}

For the first step in the above proof sketch, we can use the same Lyapunov argument presented in \cite{joaoACC}. 
Such an argument has also appeared in \cite{zhang2019policyc}.
This is formalized as follows.

\begin{lemma}\label{lemma:step_stable} 
Suppose \(\hat{K}\in\mathcal{K}\).  
Then the one-step update $\hat{K}'$ obtained from the natural policy gradient method~\eqref{eq:npgd} will also be in $\mathcal{K}$ if  the following bound holds
\begin{equation*}
 \tilde{\eta} \leq \frac{1}{2 \| \hat{R} + \gamma \hat{B}^T \EPhatnext \hat{B} \|}.
\end{equation*}
\end{lemma}
\begin{proof}
The controller $\hat{K}'$ stabilizes the $\sqrt{\gamma}$-scaled system in the mean-square sense if and only if there exists matrices \( \{Y_i\} \succ 0 \) such that
\begin{equation}
\label{eq:LMI}
\gamma (A_i - B_i K_i')^T \left[ \sum_{j \in \Omega} p_{ij} Y_j \right] (A_i - B_i K_i') - Y_i \prec 0, \quad \forall i \in \Omega
\end{equation}
We will show that the above condition can be satisfied by setting $Y_i=P_i$ where $P_i$ solves the MJLS Lyapunov equation $\gamma(A_i - B_i K_i)^T \EPKnext (A_i - B_i K_i) + Q_i + K_i^T R_i K_i = P_i$. 
Notice the existence of $P_i$ is guaranteed by the assumption $\hat{K}\in \mathcal{K}$.
Denote $\deltaK_i:= K_i-K_i'$. 
The Lyapunov equation for $P_i$ can be rewritten as $\gamma (A_i\! -\! B_i K_i'\! -\! B_i \deltaKi )^T \EPKnext (A_i\! -\! B_i K_i'\! -\! B_i \deltaKi)+ Q_i + (K_i' + \deltaKi)^T R_i (K_i' + \deltaKi) = P_i $. 
We can further manipulate this equation as 
\begin{align*}
\gamma(A_i - B_i K_i')^T \EPKnext (A_i - B_i K_i') - P_i =& -\left( Q_i + (K_i')^T R_i K_i'\right) \\
& - \left(\deltaKi^T R_i \deltaKi + \gamma \deltaKi^TB_i^T \EPKnext B_i \deltaKi \right) \\
& - \deltaKi^T \left( R_i K_i' - \gamma B_i^T \EPKnext (A_i - B_i K_i') \right) \\
& - \left( R_i K_i' - \gamma B_i^T \EPKnext (A_i - B_i K_i') \right)^T \deltaKi
\end{align*}
Notice $R_i$, $\EPKnext$, and $Q_i$ are positive definite. Hence the sum of the first two terms on the right hand side are negative definite.
Next, we will show that the last two terms are also negative semidefinite, and this will ensure that $Y_i=P_i$ provides a solution for the inequality condition \eqref{eq:LMI}.
Note that  \(\deltaKi = 2\tilde{\eta} L_i(\hat{K}) \).
 We can make the following calculations:
\begin{align*}
&\deltaKi^T \left( R_i K_i' - \gamma B_i^T \EPKnext (A_i - B_i K_i') \right) \\= & \deltaKi^T \left( (R_i +  \gamma B_i^T \EPKnext B_i) K_i' - \gamma B_i^T \EPKnext A_i \right) \\
= &\deltaKi^T \left( (R_i +  \gamma B_i^T \EPKnext B_i) \left(K_i - \deltaKi \right)  - \gamma B_i^T \EPKnext A_i \right) \\
= &2\tilde{\eta} L_i(\hat{K}) \left( L_i(\hat{K}) - 2\tilde{\eta} \left(R_i + \gamma B_i^T \EPKnext B_i\right) L_i(\hat{K}) \right) \\
=& 2\tilde{\eta} L_i(\hat{K}) \left(I - 2\tilde{\eta} \left(R_i + \gamma B_i^T \EPKnext B_i\right) \right)  L_i(\hat{K})
\end{align*}
Clearly, the above term is guaranteed to be positive semidefinite if $\tilde{\eta}$ satisfies
\begin{align*}
 \tilde{\eta} \leq \frac{1}{2\| R_i + \gamma B_i^T \EPKnext B_i \|}.
\end{align*}
Lastly, notice
$\| R_i + \gamma B_i^T \EPKnext B_i \| \leq \| \hat{R} + \gamma \hat{B}^T \EPhatnext \hat{B} \|$ for all $i$.
This leads to the desired conclusion.
\end{proof}

Several extra helper lemmas for Step 2 are now stated.

\begin{lemma} \label{lemma:bound}
Given the definitions in~\eqref{eq:lyap_markov}, the following holds
\begin{equation*}
\sum_{i\in \Omega}\| P_i^{\hat{K}} \| \leq \frac{C(\hat{K}, 0)}{\mu}.
\end{equation*}
\end{lemma}
\begin{proof}
The proof is almost identical to the proof of Lemma 7 in \citet{joaoACC}.
\end{proof}

\begin{lemma}[``Almost smoothness'']
\label{lemma:smooth}
Suppose $\hat{K}\in \mathcal{K}$ and $\hat{K}'\in \mathcal{K}$. For any fixed $\sigma$,
the cost function \(C(\hat{K}, \sigma)\)  satisfies
\begin{equation*}
C(\hat{K}', \sigma) - C(\hat{K}, \sigma) = -2\tr{\chi_{\hat{K}'}(\sigma) \Delta \hat{K}^T \hat{L}_{\hat{K}}} + \tr{\chi_{\hat{K}'}(\sigma) \Delta \hat{K}^T \left(\hat{R} + \gamma \hat{B}^T \EPhatnext \hat{B}\right)\Delta \hat{K}}
\end{equation*}
\begin{align*}
\text{where }\quad \Delta \hat{K} &= \diag{(K_1 - K_1'), \ldots, (K_{n_s} - K_{n_s}')}, \hat{L}_{\hat{K}} = \diag{L_1(\hat{K}), \ldots, L_{n_s}(\hat{K})}, \\
\EPhatnext &= \diag{\mathcal{E}_1(P^{\hat{K}}), \ldots, \mathcal{E}_{n_s}(P^{\hat{K}})}.
\end{align*}
\end{lemma}
\begin{proof}
The proof is quite similar to the proof of Lemma 5 in \citet{joaoACC}. We add a few more details here.
To simplify the equations, we use \( \phi_i = A_i - B_i K_i \) and \( \phi_i' \coloneqq A_i - B_i K_i' \).
Recall that
\begin{equation} 
C(\hat K, \sigma) = \Expx{ \sum_{i\in\Omega } \rho_i \left(x_0^T  P_i^{\hat{K}} x_0 + z_i^{\hat{K}}\right)},
\end{equation}
where \( z_i^{\hat{K}}\) is calculated from \eqref{eq:z_i}.
Using the cost function definition, we have
\begin{align}
\label{eq:Cdiff}
\begin{split}
C(\hat{K}', \sigma') - C(\hat{K}, \sigma)
	=\sum_{i\in\Omega} \tr{(P^{\hat{K}'}_i - P^{\hat{K}}_i) X_i'(0)} + \sum_{i\in\Omega } \1{\omega_0 = i}(z_i^{\hat{K}'} - z_i^{\hat{K}})
\end{split}
\end{align}
where $X_i'(t)=\Exp{x_t x_t^T \1{\omega_t = i}}$ where $x_t$ is generated under the policy $\hat{K}'$.
Now we develop a formula for $(P_i^{\hat{K}'} - P_i^{\hat{K}})$.
Based on (\ref{eq:lyap_markov}), we have $
P^{\hat{K}'}_i = \gamma (\phi_i')^T \Enext{P^{\hat{K}'}} \phi_i' + Q_i + (K_i')^T R_i K_i'$.  Using this, we can directly show
\begin{align*}
P_i^{\hat{K}'} - P_i^{\hat{K}} &= \gamma (\phi_i')^T \Enext{P^{\hat{K}'}} \phi_i' + Q_i + (K_i')^T R_i K_i' - P_i^{\hat{K}} \\
 &= \gamma (\phi_i')^T \left( \Enext{P^{\hat{K}'}} - \EPKnext \right) \phi_i' + (K_i - K_i')^T (R_i + \gamma B_i^T \EPKnext B_i)(K_i - K_i')\\
&  \quad - (K_i - K_i')^T\left(R_i K_i - \gamma B_i^T \EPKnext \phi_i\right) - \left(R_i K_i - \gamma B_i^T \EPKnext \phi_i\right)^T(K_i - K_i') 
\end{align*}
Now we can substitute the above formula into \eqref{eq:Cdiff} and  show
\begin{align*}
&\sum_{i\in\Omega} \tr{(P^{\hat{K}'}_i - P^{\hat{K}}_i) X_i'(0)} + \sum_{i\in\Omega } \1{\omega_0 = i}(z_i^{\hat{K}'} - z_i^{\hat{K}})\\=&-2\tr{\diag{X_i(0)} \Delta \hat{K}^T \hat{L}_{\hat{K}}} + \tr{\diag{X_i(0)} \Delta \hat{K}^T \left(\hat{R} + \gamma \hat{B}^T \EPhatnext \hat{B}\right)\Delta \hat{K}}\\
&+ \gamma\left(\sum_{i\in\Omega} \tr{(P^{\hat{K}'}_i - P^{\hat{K}}_i) X_i'(1)} + \sum_{i\in\Omega } \1{\omega_1 = i}(z_i^{\hat{K}'} - z_i^{\hat{K}})\right)
\end{align*}
Actually it is straightforward to extend the above formula for any $t$:
\begin{align*}
&\sum_{i\in\Omega} \tr{(P^{\hat{K}'}_i - P^{\hat{K}}_i) X_i'(t)} + \sum_{i\in\Omega } \1{\omega_t = i}(z_i^{\hat{K}'} - z_i^{\hat{K}})\\=&-2\tr{\diag{X_i(t)} \Delta \hat{K}^T \hat{L}_{\hat{K}}} + \tr{\diag{X_i(t)} \Delta \hat{K}^T \left(\hat{R} + \gamma \hat{B}^T \EPhatnext \hat{B}\right)\Delta \hat{K}}\\
&+ \gamma\left(\sum_{i\in\Omega} \tr{(P^{\hat{K}'}_i - P^{\hat{K}}_i) X_i'(t+1)} + \sum_{i\in\Omega } \1{\omega_{t+1} = i}(z_i^{\hat{K}'} - z_i^{\hat{K}})\right)
\end{align*}
Therefore, we can iterate the above formula from $t=0$ to $\infty$ to obtain the desired conclusion.
\end{proof}

\begin{lemma}
[Natural Gradient Domination]
\label{lemma:grad_dom}
Suppose  \(\hat{K}\in\mathcal{K}\). Let \(\hat{K}^*\) be the optimal policy. Then 
\(C(\hat{K}, 0) - C(\hat{K}^*, 0) \leq \frac{\| \chi_{\hat{K}^*}(0) \|}{\sigma_{\min}(\hat{R})} \tr{\hat{L}_{\hat{K}}^T \hat{L}_{\hat{K}}}\).
\end{lemma}
\begin{proof}
The proof is almost identical to the proof of Lemma 6 in \cite{joaoACC}. Just notice that the almost smoothness property holds for $\sigma=0$, and the coefficient $\chi_{\hat{K}^*}$ in the above bound is evaluated at $\sigma=0$.
\end{proof}

Now we can analyze the one-step progress.

\begin{lemma}\label{lemma:one_step}
Suppose $\hat{K}\in \mathcal{K}$, and $\mu>0$.
If $\hat{K}' = \hat{K} - \tilde{\eta} \nabla_{\hat{K}} C(\hat{K},\sigma)\chi_{\hat{K}}(\sigma)^{-1}$ with a scaled stepsize satisfying $0\le  \tilde{\eta} \le \frac{1}{2\| \hat{R} + \gamma \hat{B}^T \EPhatnext \hat{B} \|}$,
then the following inequality holds
\begin{align*}
C(\hat{K}', 0) - C(\hat{K}^*, 0) &\leq \left( 1 - 2\tilde{\eta} \mu\frac{  \sigma_{\min}(\hat{R})}{\| \chi_{\hat{K}^*}(0) \|} \right) \left( C(\hat{K}, 0) - C(\hat{K}^*, 0) \right).
\end{align*}
\end{lemma}
\begin{proof}
The proof is identical to the proof of Lemma 9 in \cite{joaoACC}.
\end{proof}

Finally, in Step 3, induction can be used to prove Theorem~\ref{thm:conv}. This is identical to the induction proof presented in the end of \cite{joaoACC}. We omit the details here.

\section{More Discussions on Algorithm 1}
Here we provide more details on the equations used in Algorithm~\ref{alg:alg1}.
Recall that we sample our input actions from a Gaussian distribution as $u_t\sim \N(-K_{\omega_t} x_t, \sigma^2 I)$. 
Then our likelihood of choosing a given action is
\begin{equation}
\pi_\theta(u_t | x_t, \omega_t) = \frac{1}{\sigma^k \sqrt{(2 \pi)^{k}}} \exp \left( -\frac{1}{2 \sigma^2}  (u_t + K_{\omega_t} x_t)^T (u_t + K_{\omega_t} x_t) \right),
\end{equation}
and it follows that
\begin{equation}\label{eq:log_like}
\log \pi_\theta(u_t | x_t, \omega_t) = -\frac{k}{2}\log 2\pi -k \log \sigma - \frac{1}{2\sigma^2} (u_t + K_{\omega_t} x_t)^T (u_t + K_{\omega_t} x_t).
\end{equation}

The derivative of the log-likelihood function with respect to control parameters at mode $i$ is
\begin{equation}
\frac{\partial}{\partial K_i} \log \pi_\theta(u_t | x_t, \omega_t) = -\frac{1}{\sigma^2}  (u_t + K_{\omega_t} x_t) x_t^T \1{\omega_t = i}.
\end{equation}
Due to the shape of $\hat{K}$, we can also write the above equation using basis vectors and the Kronecker product, which leads to the formula for $\hat{G}_t$ in Algorithm 1.
We can obtain $\hat{S}_t$ by taking the derivative of the log-likelihood with respect to $\sigma$.

\section{Implementation}

The code used to generate the simulations in Section~\ref{sec:DataNPG} are available at \url{https://github.com/jpjporto/MJLS_Learning}

\end{document}